\documentclass[12pt,a4paper]{article}

\usepackage{theorem}
\usepackage{enumerate}
\usepackage{amsmath}
\usepackage{latexsym}
\usepackage{amssymb}
\usepackage{amsfonts}
\usepackage{graphicx, epsfig,epstopdf, subfigure}
\usepackage{CJK}
\usepackage[varg]{txfonts}
\usepackage{here}
\usepackage{multirow}
\usepackage{array}
\usepackage[T1]{fontenc}
\usepackage[utf8]{inputenc}
\usepackage{authblk}

\newcounter{RomanNumber}

\theorembodyfont{\normalfont\itshape}
\newtheorem{thm}{Theorem}[section]

\newtheorem{corollary}[thm]{Corollary}

\newtheorem{lem}[thm]{Lemma}

\newtheorem{clm}[thm]{Claim}

\newtheorem{problem}[thm]{Problem}
\newtheorem{observation}[thm]{Observation}

\theoremstyle{remark}

{\theorembodyfont{\upshape}
}

\newcommand{\qed}{{$\quad\square$\vspace{2mm}}}

\bfseries\normalfont

\begin{document}

\title{Packing a number of copies of a $(p,\,q)$-graph\thanks{The author's work is supported by NNSF of China (No.11671232, 12071260) and NSF of Shandong Province (No. ZR2017MA018).}}

\author{Yun Wang, Jin Yan\thanks{Corresponding author: Jin Yan, Email:  yanj@sdu.edu.cn} \unskip\\
School of Mathematics, Shandong University, Jinan 250100, China}

\date{}
\maketitle
\date{}
\maketitle

\begin{abstract}
Let $k,p,q$ be three positive integers. A graph $G$ with order $n$ is said to be $k$-placeable if there are $k$ edge disjoint copies of $G$ in the complete graph on $n$ vertices.  A $(p,\,q)$-graph is a graph of order $p$ with $q$ edges. Packing results have proved useful in the study of the complexity of graph properties. Bollob\'{a}s et al.  investigated the $k$-placeability of $(n,\,n-2)$-graphs  and $(n,\,n-1)$-graphs with $k=2$ and $k=3$. Motivated by their results, this paper characterizes $(n,\,n-1)$-graphs with girth at least $9$ which are $4$-placeable. We also consider the $k$-placeability of $(n,\,n+1)$-graphs and 2-factors.
\end{abstract}


\section{Introduction}  \label{section 1}

This paper considers only finite simple graphs and uses standard terminology and notation from \cite{Bondy1976Graph} except as indicated. For any  graph $G$, we denote by $V(G)$ (resp. $E(G)$) the vertex set (resp. the edge set). The \emph{maximum degree} (resp. \emph{minimum degree}) of $G$ is denoted by $\Delta(G)$ (resp.  $\delta(G)$). For two  graphs $H_1$ and $H_2$, we use $H_1\uplus H_2$ to represent the vertex disjoint union of $H_1$ and $H_2$.  A \emph{2-factor} is a graph whose components are all cycles. Let $K_n$ be the complete graph of order $n$.  A path, cycle and star with order $n$ is denoted by $P_n, C_n$ and $S_n$, respectively. The tree $S_{a}^b$, of order $a+b$, is obtained from star $S_{a}$ by inserting $b$ vertices into an edge of $S_{a}$. The girth of $G$, i.e. the length of a shortest cycle of $G$, is denoted by $g(G)$. A vertex of $G$ with degree $1$ is a \emph{leaf}.

We say that a $k$-tuple $\Phi=(\phi_1,\phi_2,\ldots, \phi_k)$ is a \emph{$k$-placement} of a graph $G$ on $n$ vertices if for each $i$, $\phi_i$ is a permutation: $ V(G)\rightarrow V(K_n)$ such that $E(\phi_i(G))\cap E(\phi_j(G)) = \emptyset$ for $i \neq j$. If $G$ has a $k$-placement, then $G$ is \emph{$k$-placeable}.  A graph with $p$ vertices and $q$ edges is called a $(p,\,q)$-graph.

Packing results have proved useful in the study of the complexity of graph properties \cite{bollobas1978}. It is worthwhile mentioning that the packing problem  is NP-complete.  Interestingly, the 2-placeability of  $(n,\,n-2)$-graphs was independently solved by several groups of researchers almost at the same time, see Bollob\'{a}s and Eldridge \cite{be1978}, Burns and Schuster \cite{bs1977} and Sauer and Spencer \cite{ss1978}.  As for 2-placeability of $(n,\,n-1)$-graphs, the first result was given by Hedetniemi, Hedetniemi and Slater in \cite{hhs1981}. They solved the problem when the $(n,\,n-1)$-graph is a tree. Later, Burns and Schuster \cite{bs1978} and Yap \cite{survey1988} generalized this result to all $(n,\,n-1)$-graphs and proved the following.

\begin{thm}\label{theorem 11} (\cite{bs1978})  Let $G$ be a graph of order $n$ with at most $n-1$ edges. Then either $G$ is 2-placeable or $G$ is isomorphic to one of the following graphs: $S_n$, $ S_{n-3}\uplus C_3 (n\geq 8)$, $K_1\uplus 2C_3$, $K_1\uplus C_4$, $K_1\uplus C_3$, $K_2\uplus C_3$.
\end{thm}

On packing three graphs, Wo\'{z}niak and Wojda \cite{wowo1993} proved that nearly all $(n,\,n-2)$-graphs are 3-placeable. Motivated by this result, Wang and Sauer \cite{ws1993,ws1996} considered the 3-placeability of connected $(n,\,n-1)$-graphs (each of them is a tree) and disconnected $(n,\,n-1)$-graphs, respectively. They proved the following theorem.

\begin{thm}\label{theorem 12} (\cite{ws1996})  Let $G$ be an $(n,\,n-1)$-graph with $g(G)\geq 5$ and order $n\geq 6$. Then $G$ is 3-placeable if and only if $G$ is not  isomorphic to one of the following graphs: $S_n$, $C_5\uplus K_1$, $S_4^2$, $S_{n-1}^1$.
\end{thm}

Since $e(K_n)=\frac{n(n-1)}{2}\geq k(n-1)$ holds only if $n\geq 2k$, an $(n,\,n-1)$-graph with $2\leq n\leq 2k-1$ is not $k$-placeable (the case $n=1$ is trivial). Also, no connected graph with $\Delta(G)\geq n-k+1$ is $k$-placeable. A natural and interesting problem arises.

\begin{problem}\label{problem 1} Let $G$ be an $(n,\, n-1)$-graph with $n\geq 2k$ and $\Delta(G)\leq n-k$. Is $G$ $k$-placeable? If $G$ is not $k$-placeable, can we characterize its structure?
\end{problem}

Actually, \.{Z}ak \cite{zak2011} considered $k$-placeability of sparse graphs. He proved that a graph $G$ of order $n\geq2(k-1)^3$ is $k$-placeable if $|E(G)|\leq n-2(k-1)^3$. In general, the problem of $k$-placeability is more difficult for dense graphs than for sparse graphs. We consider the case $k=4$ with large girth of Problem \ref{problem 1} in this paper. Note that Theorems \ref{theorem 11} and  \ref{theorem 12} imply that some graphs containing small cycles are not $2$-placeable and $3$-placeable. Moreover, the graph $C_7\uplus K_1$ is not 4-placeable because the degree of each vertex in $K_8$ is odd. For this reason and in order to make it clear for readers to understand the tricks in our paper, we investigate $(n,\, n-1)$-graphs with $g(G)\geq 9$, rather than $g(G)\geq 8$ or $g(G)\geq 7$. Let $W$ be the set of graphs depicted in Fig. \ref{fig1}. The following theorem is our main result.

\begin{thm}\label{theorem 2} Let $G$ be an $(n,\, n-1)$-graph with $g(G)\geq 9$ and order $n\geq 8$. Then $G$ is $4$-placeable if and only if $\Delta(G)\leq n-4$ and $G\notin W$.
\end{thm}

\begin{figure}[H]
    \centering      
    \includegraphics[scale=0.4]{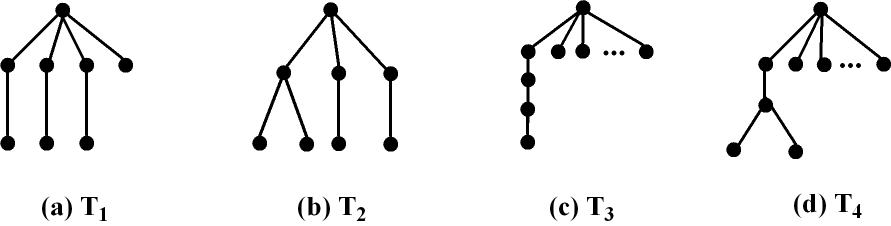}   
\caption{$W$}
\label{fig1} 
\end{figure}
Haler and Wang solved the case that $G$ is connected by proving the following.

\begin{thm}\label{lem 1111}(\cite{halerwang2014four-p}) A tree $T$  of order $n\geq 8$ is $4$-placeable if and only if $\Delta(T)\leq n-4$ and $T\notin W$.
\end{thm}

Theorems mentioned above focus on $k$ copies of a given graph. In fact, there are also some results concerning the packing of different graphs, see \cite{kmsw2001,wo2004} and other kinds of packing problems, see, e.g. \cite{list2015} (a list version of packing) and \cite{degree2016}. Several  conjectures on packing problems have appeared in the literature. Let us mention two of the best-known and interesting, which are still open.

One is the Tree Packing Conjecture (TPC) posed by Gy\'{a}rf\'{a}s \cite{Gyarfas1978}: Any set of $n-1$ trees $T_2,T_3,\ldots,T_n$ such that every $T_i$ of order $i$ has a packing into $K_n$.  Many partial results of TPC discovered by a number of researchers, see, e.g., the papers  \cite{balogh2012} by Balogh and Palmer, \cite{bollobas1983} by Bollob\'{a}s, \cite{zak2017} by \.{Z}ak  and \cite{zaks1977} by Zaks and Liu.  It is worthwhile mentioning that it is not known when the biggest four can be non-stars, so the state-of-the art knowledge is three trees \cite{hbk1987} while by \cite{halerwang2014four-p} one can pack four copies of a tree. In particular, \cite{joos2019} is a recent breakthrough in this topic, which showed that the TPC holds for all bounded degree trees. In \cite{Gyarfas2014}, Gy\'{a}rf\'{a}s showed that if such $n-1$ trees has a packing into $K_n$, then they can be also packed into any $n$-chromatic graph. Instead of packing trees into the complete graph, Gerbner et al. \cite{Gerbner2012} and Hobbs et al. \cite{hbk1987} conjectured that the trees $T_2,T_3,\ldots,T_n$ have a packing into $k$-chromatic graphs and the complete bipartite graph $K_{n-1,\lceil\frac{n}{2}\rceil}$, respectively.  For more on the latter conjecture, see \cite{Caro1990,hobbs1981,Yuster1997,zaks1977}.

Another interesting packing problem is the Bollob\'{a}s-Eldridge-Catlin (BEC) conjecture \cite{be1978,Catlin1976}, which considers packing two graphs  $G_1$ and $G_2$ with $(\Delta(G_1)+1)(\Delta(G_2)+1)\leq n+1$ into $K_n$. The conjecture is correct for $\Delta(G_1) = 2$ and $\Delta(G_1) = 3$ and $n$ sufficiently large \cite{ab1993,Csaba2003}, some classes of graphs \cite{ddege2008,Csaba2007}, and graphs with restrictions on degree or girth \cite{girth2017,Cames2018}.

The rest of the paper is organized as follows: The aim of Section \ref{section 2} is to prepare some notation and terminology used in the paper. In Section \ref{section 3},  we show some lemmas which are useful in the proof of Theorem \ref{theorem 2}. In Section \ref{section 4}, the proof of Theorem \ref{theorem 2} is given.
\section{Notation} \label{section 2}

We use the following notation throughout the paper.  The \emph{order} (resp. \emph{size}) of a graph $G$ is defined by $v(G)$ (resp. $e(G)$). That is, $v(G)=|V(G)|$, $e(G)=|E(G)|$. A \emph{degree} of a vertex $v\in V(G)$ is denoted by $d_G(v)$.  A vertex of degree at least two  adjacent to  a leaf is called a \emph{node}. For a subset $U$ of $V(G)$, the subgraph in $G$ induced by $U$ is denoted by $G[U]$, and let $G -U = G[V(G)\backslash U]$. \emph{A vertex $v$ of $G$ is $k$-placed} (resp. \emph{$k$-fixed}) by $\Phi$ if for each $i\neq j\in\{1,2,\ldots, k\}$, $\phi_i(v)\neq \phi_j(v)$ (resp. $\phi_i(v)=\phi_j(v)$). Moreover, if every vertex of $G$ is $k$-placed, then $\Phi$ is \emph{dispersed}. \emph{An edge $e$ is $k$-placed} by $\Phi$ if the set of edges $\{\phi_i(e): i=1,2,\ldots,k\}$ are vertex disjoint.

\textbullet \,\, A \emph{double} \emph{lasso} $D(l,s,t)$  consists of a path $v_1v_2\cdots v_l$ with additional edges $v_1v_s$ and $v_lv_{l-t+1}$, where $3\leq s\leq l, 3\leq t<l$.

\textbullet \,\, A \emph{lasso} $L(l,s)$ is obtained by deleting the edge $v_lv_{l-t+1}$ from $D(l,s,t)$. Clearly, $L(l,l)\cong C_l$.

\textbullet \,\, The graph obtained by replacing each leaf of $S_{t+1}$ with a path $P_{n_i}$ is \emph{$Q(n_1,\ldots, n_t)$}. It will be assumed that $1\leq n_1\leq n_2\leq \cdots\leq n_t$. Write $v$ as the center of $S_{t+1}$ and write $P_{n_i}=v_1^iv_2^i\cdots v_{n_i}^i$, where $v$ and $v_1^i$ are adjacent for each $i$.

Observe  that each connected $(n,n-1)$-graph is a tree, each connected $(n,n)$-graph contains a cycle and each connected  $(n,n+1)$-graph contains a double lasso.
\section{Preliminary results} \label{section 3}

\begin{observation}\label{observation 1}
Let $G$ be a graph and let $U$ be a set of some leaves. If $G-U$ has a $k$-placement  such that each vertex in $N_G(U)$ is $k$-placed, then $G$ is $k$-placeable.\end{observation}

Let $A$ and $B$ be two vertex disjoint induced subgraphs of $G$ and let $U\subseteq V(G)$ be an independent set such that either $V(A),V(B),U$ is a partition of $V(G)$ or $V(A),V(B)$ is a partition of $V(G)$ with $U\subseteq V(A)$ or  $U\subseteq V(B)$. We allow $U$ to be an empty set. A graph $G$ is an  \emph{($A,U,B$)-structure} graph if

(i) at most one vertex $a\in V(A)-U$ has neighbors in $B-U$, and  for each $u\in U$, $|N_A(u)|\leq 1$ if $U\cap V(A)=\emptyset$ and $|N_B(u)|\leq 1$ if $U\cap V(B)=\emptyset$,

(ii) each of $A$ and $B$ has a $k$-placement such that the vertices in $N_G(U)$ and $a$  are $k$-placed, and each vertex in $U$ is $k$-fixed.

\begin{figure}[H]
    \centering      
    \includegraphics[scale=0.3]{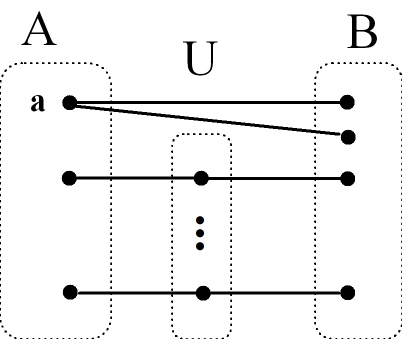}   
\caption{A partition of $G$}
\label{fig2} 
\end{figure}

\begin{lem}\label{lem 4}
If $G$ is an  ($A,U,B$)-structure graph, then $G$ is $k$-placeable.
\end{lem}
\begin{proof}
Suppose $v(G)=n$. First, we consider the case $U\cap V(A\uplus B)=\emptyset$ (as shown in Fig. \ref{fig2}). Let $\Phi(A),\Phi(U)$ and $\Phi(B)$ be the $k$-placements of $A,U$ and $B$, respectively. Now we construct a $4$-placement of $G$. First divide $K_n$ into three vertex disjoint subgraphs $K_{v(A)}$, $K_{v(U)}$ and  $K_{v(B)}$, and  put $\Phi(A)$, $\Phi(U)$ and $\Phi(B)$ into these three subgraphs, respectively. Recall that $G$ is an  ($A,U,B$)-structure graph, that is, $A,B$ and $U$ satisfy the condition (i) of the definition above. Since the vertices of $N_G(U)$ and $a$ are $k$-placed and each vertex in $U$ is $k$-fixed, it is not difficult to check that the $k$ copies of edges between $A, B$ and $U$ are edge disjoint. That is, one can obtain a $k$-placement $\Phi(G)$ of $G$, where  $\Phi(G)=\Phi(A)\cup \Phi(U) \cup \Phi(B)$.

In the cases $U\subseteq V(A)$ and $U\subseteq V(B)$, divide $K_n$ into two vertex disjoint subgraphs $K_{v(A)}$ and  $K_{v(B)}$. Put the $k$-placements of $A$ and $B$ into these two subgraphs, a $k$-placement of $G$ can be obtained similarly.\qed\end{proof}

The following interesting lemma is a key lemma, which improves Lemma 7 in \cite{zak2011} of \.{Z}ak.
\begin{lem}\label{lem 2}
Let $U\subseteq V(G)$ be a set of $k$ leaves such that the vertices in $U$  have distinct neighbors and let $V=N_G(U)$. Suppose that $G-U$ is $k$-placeable in $K_{n-k}$, where $n=|V(G)|$. Let $\phi_i(V)=V_i$ for each $1\leq i\leq k$. Then in $K_n$ there exist $k$ edge disjoint  matchings $M_1,M_2,\ldots,M_k$ that match $V_1,V_2,\ldots,V_k$  to $U$, respectively. That is, $G$ is also $k$-placeable.\end{lem}

\begin{proof} Let $U=\{u_1,u_2,\ldots, u_k\}$ and $W=V_1\cup V_2\cup\cdots \cup V_k$. Note that $0\leq |V_i\cap V_j|\leq k$ for each $1\leq i,j\leq k$ as each $V_i$ is a set of $k$ vertices.  Let $B(U,W)$ be a bipartite graph with partition classes $U$ and $W$ such that every vertex in $V_i$ is adjacent to $u_i$ for each $1\leq i\leq k$. It is well-known that the edge chromatic number of a bipartite graph $B$  equals the maximum degree of $B$ (K\"{o}nig's theorem). So the edges of the bipartite graph $B(U,W)$ can be colored with exactly $k$ colors such that adjacent edges are colored distinct colors. Assume these $k$ colors are $c_1,c_2,\ldots, c_k$.

For each $i$, we construct a perfect matching $M_i$ between $V_i$ and $U$ as follows: for each $w \in V_i$, if there is an edge incident with $w$ colored with $c_j$ in $B(U,W)$, then add $u_jw$ to $M_i$ in $K_n$. Since the $k$ edges of $B(U,W)$ incident to $u_i$ have distinct colours, $M_i$ is indeed a perfect matching between $V_i$ and $U$. Moreover, the $k$ matchings $M_1,\ldots, M_k$ are edge-disjoint. If not, assume $u_iw\in M_j\cap M_l$. Then clearly $w\in V_j \cap V_l$ and then $u_jw,u_lw\in E(B(U,W))$. Moreover, the edges $u_jw$ and $u_lw$ are colored by the same color $c_i$ in $B(U,W)$, a contradiction.\qed\end{proof}

Let $t,k$ be two positive integers with $t\geq 2k$. Now we construct a dispersed $k$-placement of the path $P_t=v_1,\ldots,v_t$ in $K_t$. Let $V(K_t)=\{u_1,\ldots,u_t\}$. Define
\begin{equation}\label{eq 1}
\phi_i(P_t)=u_iu_{t-1+i}u_{i+1}u_{t-2+i}\cdots u_{\lfloor \frac{t}{2}\rfloor+i} \mbox{ for } i=1,2,\ldots,k,
\end{equation}
where the subscripts of the $u_j$'s are taken modulo $t$ in $\{1,2,\ldots,t\}$ (see Fig. \ref{fig9} (a), $\phi_i$ with $2\leq i\leq k$ can be obtained by rotating  $\phi_{i-1}$  one `unit' in the direction of the arrow. One can check that $\phi_1,\phi_2,\ldots,\phi_{k-1}$ and $\phi_k$ are pairwise edge-disjoint because $t\geq 2k$).

In order to show the property of the $k$-placement of the path $P_t$ in $K_t$, we use a table to exhibit $\Phi(P_t)$ (see Fig. \ref{fig9} (b)), where the vertex $u_i$ is replaced by $i$.

\begin{figure}[H]
\centering    
    \includegraphics[scale=0.3]{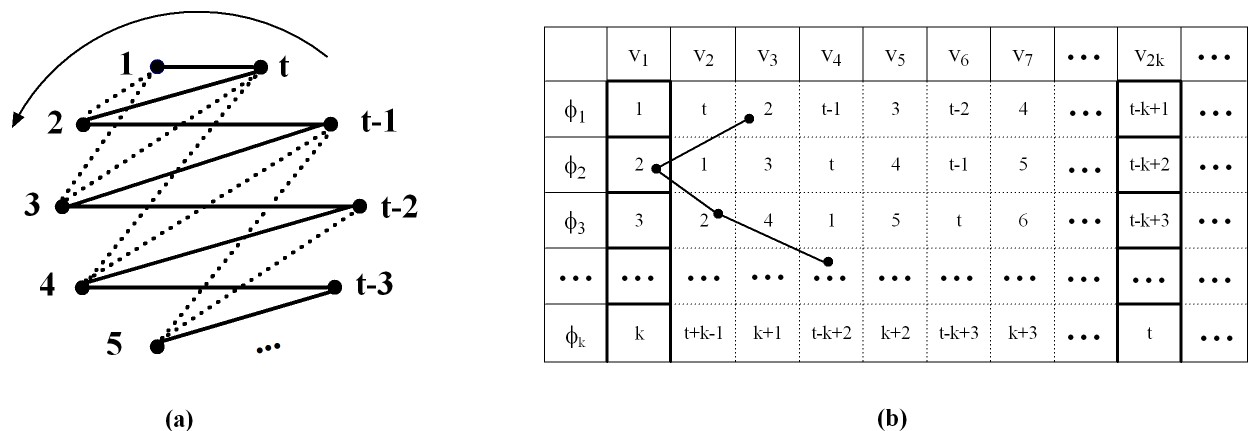}   
\caption{$\Phi(P_t)$} 
\label{fig9}  
\end{figure}

In the table,  since $t\geq 2k$,  the $k$-placements of $v_1$ and $v_{2k}$, i.e., the elements $1,2,\ldots,k, t-k+1,t-k+2,\ldots,t$ are pairwise distinct.  More generally, the following important property holds. \begin{equation}\label{eq 2}
\begin{split}
\text{ The } k\text{-placements of } v_a \text{ and } v_{b} \text{ with } |a-b|\geq 2k-1 \text{ are pairwise distinct.}
\end{split}
\end{equation}

To see this,  observe that each vertex $u_i$ arises exactly $k$ times in the table. Furthermore, suppose that columns $s$ and $t$ are the first and last column in which $u_i$ appears, respectively,  then $|s-t|\leq2k-2$ (for example, see $u_2$s in the table).  Thus if two vertices $v_a$ and $v_b$ has `large' distance ($|a-b|\geq2k-1$) on the path $P_t$, then the $k$-placements of $v_a$ and $v_b$ are $2k$ distinct vertices, that is, $\phi_i(v_a)\neq \phi_j(v_b)$ for each $1\leq i,j\leq k$.

By the construction of the dispersed $k$-placement of a path and (\ref{eq 2}), the following lemma follows immediately.
\begin{lem}\label{lemma 2.1} The path $P=v_1\cdots v_l$ with $l\geq 2k$  has a dispersed $k$-placement. Moreover, the $2k$ elements $\phi_i(a),\phi_j(b)$ with $|a-b|\geq 2k-1$ and  $1\leq i,j\leq k$ are pairwise distinct.
\end{lem}

\begin{lem}\label{lemma 2.11} Let $k$ be an integer with $k\geq 4$. Label $C_l,L(l,s)$ and $D(l,s,t)$ as defined. Then the following statements are true.

(i) The cycle $C_l$ with $l\geq 2k+1$ has a $k$-placement such that all vertices except $v_1$ are k-placed.

(ii) The lasso $L(l,s)$  with $s\geq 2k+1$ has a $k$-placement such that all vertices except $v_1$ are $k$-placed.

(iii) The double lasso $D(l,s,t)$ with $s\geq 2k+1$, $t\geq 2k+1$ has a $k$-placement such that all vertices except $v_1,v_l$ are $k$-placed.
\end{lem}

\begin{proof} It suffices to prove (ii) and (iii) as $C_l=L(l,l)$. To prove (iii), let  $P_{l-2}=v_2v_3\cdots v_{l-1}$ and define a dispersed $k$-placement  $\Phi(P_{l-2})$ as (\ref{eq 1}). Moreover, (\ref{eq 2}) implies  that $\phi_i(v_2)$, $\phi_j(v_{s})$ (and $\phi_i(v_{l-1})$, $\phi_j(v_{l-t+1})$) are pairwise distinct for $i,j\in\{1,2,\ldots,k\}$. Adding  vertices $v_1$, $v_l$ and edges $v_1\phi_i(v_2)$, $v_1\phi_i(v_s)$, $v_l\phi_i(v_{l-1})$ and $v_l\phi_i(v_{l-t+1})$ for $1\leq i\leq k$, we obtain a $k$-placement of $D(l,s,t)$ such that all vertices except $v_1,v_l$ are $k$-placed.

The proof of (ii) is similar. Let  $P_{l-1}=v_2v_3\cdots v_{l}$. Construct a dispersed $k$-placement of $P_{l-1}$ ($\Phi(P_{l-1})$) as (\ref{eq 1})  and then add a vertex $v_1$ and  edges $v_1\phi_i(v_2)$, $v_1\phi_i(v_s)$, we obtain a $k$-placement of $L(l,s)$ such that all vertices except $v_1$ are $k$-placed.\qed\end{proof}

Using Lemma \ref{lemma 2.11}, we claim that  $(n,\,n+1)$-graphs with large girth and  minimum degree at least 2 are $k$-placeable.

\begin{corollary}\label{theorem 13} Let $k$ be an integer with $k\geq 4$. If $G$ is an $(n,\,n+1)$-graph with $g(G)\geq 2k+1$ and $\delta(G)\geq 2$, then $G$ is $k$-placeable.\end{corollary}
\begin{proof} Clearly, $\Delta(G)\geq 3$. From Lemma \ref{lemma 2.11} (i), we assume  that $G$ contains no cycle as a component. Since $\delta(G)\geq 2$ and $\sum _{v\in V(G)}d_G(v)=2n+2\geq \Delta(G)+2(n-1)$, we derive that $3\leq\Delta(G)\leq 4$. Then $G\cong D(n,s,t)$ with $s\geq 2k+1,t\geq 2k+1$ and $3\leq s,t<n$. And if $\Delta(G)=4$, then $s=n-t+1$. Lemma \ref{lemma 2.11} (iii) implies that the corollary holds.\qed\end{proof}

\begin{lem}\label{lemma 2.7} The following statements are true.

(i) (\cite{halerwang2014four-p}) Each of $Q(2,2,3)$ and $Q(2,2,2,2)$ has a dispersed $4$-placement.

(ii)  $C_{l}\uplus Q(2,2,2)$ with $l\geq 9$  has a dispersed 4-placement.

(iii)  Let $s,l_i$ $(1\leq i\leq s)$ be positive integers with $s\geq 2$. If $\sum_{i=1}^sl_i\geq 3$, then $\uplus_{i=1}^s P_{l_i}\uplus K_1$ has a dispersed $4$-placement.

(iv) $2K_1\uplus Q(n_1,n_2,n_3)$ with $2\leq n_1\leq n_2\leq n_3$ has a $4$-placement such that $2K_1$ and all nodes of $Q(n_1,n_2,n_3)$ are $4$-placed.
\end{lem}

\begin{proof}
(ii) Label $Q(2,2,2)$ as defined. Lemma \ref{lemma 2.1} and (i) imply that each of $P_{l-1}$ and $Q(2,2,3)$ has a dispersed 4-placement. Let $u,v$ be the end-vertices of $P_{l-1}$. After adding edges $\phi_i(v_3^3)\phi_i(u)$, $\phi_i(v_3^3)\phi_i(v)$ and deleting edges $\phi_i(v_3^3v_2^3)$ for $1\leq i\leq 4$, we obtain a dispersed 4-placement of $C_l\uplus Q(2,2,2)$.

(iii) It suffices to prove the case of $s=2$. If $l_1+l_2\geq 7$, Lemma \ref{lemma 2.1} implies that  $P_{l_1}\uplus P_{l_2}\uplus K_1$ has a dispersed $4$-placement. If $3\leq l_1+l_2\leq 6$, the dispersed $4$-placement  $\Phi(P_{l_1}\uplus P_{l_2}\uplus K_1)$ is exhibited in the Fig. \ref{fig7}. The one shown by bold lines is $\phi_1(P_{l_1}\uplus P_{l_2}\uplus K_1)$. For $i=2,3$, $\phi_i$ can be obtained by rotating $\phi_{i-1}$ one `unit' counterclockwise. And  one can always put $\phi_4$ as shown by thin dashed lines in Fig. \ref{fig7} to obtain a dispersed $4$-placement. For example, see Fig. \ref{fig7} (d), let $P_4=v_1v_2v_3v_4$ and put $\phi_4(v_4)$ and  $\phi_1(v_1),\phi_2(v_3)$  on the same vertex (i.e. the top vertex), then we get a dispersed $\Phi(P_4\uplus 2K_1)$.

(iv) Adding three edges between $2K_1$ and $Q(n_1,n_2,n_3)$, we obtain $L(n_1+n_2+n_3+3,n_2+n_3+3)$. If $n_2+n_3\geq 6$, then $n_3\geq 3$ and $2K_1\uplus Q(n_1,n_2,n_3)$ has a 4-placement such that all vertices except $v_{1}^3$ are 4-placed by Lemma \ref{lemma 2.11} (ii). Thus (iv) holds. Therefore, we may assume that $n_1=n_2=2$, $n_3=2$ or $n_3=3$. The graph $2K_1\uplus Q(2,2,2)$ has a dispersed 4-placement by (i) ($Q(2,2,2,2)$). Furthermore, $2K_1\uplus Q(2,2,3)$  has a  $4$-placement such that $2K_1$ and the nodes of $Q(n_1,n_2,n_3)$ are $4$-placed by Observation \ref{observation 1} ($U=\{v_3^3\}$).\qed\end{proof}
\begin{figure}[H]\centering
\subfigure[$\Phi(P_2\uplus 2K_1)$] { \label{fig:a}
\includegraphics[width=0.15\columnwidth]{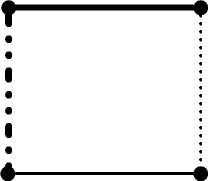}
}
\qquad
\subfigure[$\Phi(2P_2\uplus K_1)$] { \label{fig:b}
\includegraphics[width=0.15\columnwidth]{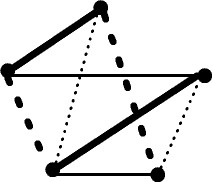}
}
\qquad
\subfigure[$\Phi(P_3\uplus 2K_1)$] { \label{fig:c}
\includegraphics[width=0.15\columnwidth]{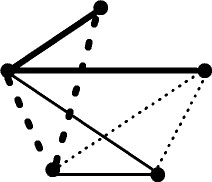}
}
\qquad
\subfigure[$\Phi(P_4\uplus 2K_1)$] { \label{fig:d}
\includegraphics[width=0.15\columnwidth]{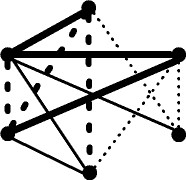}
}
\qquad
\subfigure[$\Phi(P_3\uplus P_2\uplus K_1)$] { \label{fig:e}
\includegraphics[width=0.15\columnwidth]{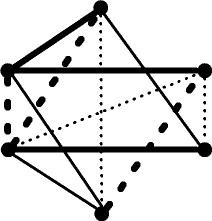}
}
\qquad
\subfigure[$\Phi(P_4\uplus P_2\uplus K_1)$] { \label{fig:f}
\includegraphics[width=0.15\columnwidth]{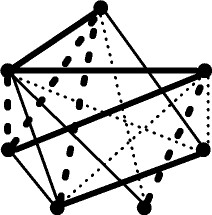}
}
\qquad
\subfigure[$\Phi(2P_3\uplus K_1)$] { \label{fig:g}
\includegraphics[width=0.15\columnwidth]{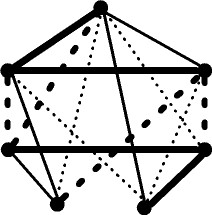}
}
\qquad
\subfigure[$\Phi(P_5\uplus 2K_1)$] { \label{fig:h}
\includegraphics[width=0.15\columnwidth]{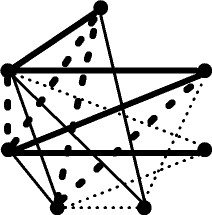}
}
\caption{$\Phi(P_{l_1}\uplus P_{l_2}\uplus K_1)$}
\label{fig7}  
\end{figure}

\.{Z}ak \cite{zak2011} proved that a graph $G$ on $n$ vertices is $k$-placeable if $2(k-1)\Delta(G)^{2}<n$. That is, every graph of order $n>8(k-1)$ with $\Delta(G)=2$ is $k$-placeable. In fact, the lower bound $8(k-1)$ can be improved  by Lemma \ref{lemma 2.11} (i) and the following Theorem \ref{thm 2.8}.

\begin{thm}\label{thm 2.8}
Let $k$ be a positive integer. A graph $G$ of order $n\geq 6k-4$ with $\Delta(G)=2$ is $k$-placeable.\end{thm}

\begin{proof} We proceed by induction on $k$. Clearly,  $G$ is $1$-placeable because the complete graph $K_n$ contains $G$.  Assume that $G$ is $s$-placeable with $1\leq s<k$. Now, we prove that $G$ is $(s+1)$-placeable. Let $H$ be the graph obtained  from $K_n$ by deleting the edges of $s$ copies of $G$. Clearly, $\delta(H)\geq n-1-2s\geq \frac{2n-1}{3}$ for $1\leq s<k$ because  $n\geq 6k-4$. So the proof of the theorem is completed by the result of Aigner and Brandt \cite{ab1993}: A graph $H$ of order $n$ with $\delta(H)\geq \frac{2n-1}{3}$  contains any graph $G$ of order at most $n$ with $\Delta(G)=2$.\qed\end{proof}

It should be noted that a 2-factor is an $(n,\,n)$-graph. The following corollary follows immediately by Lemma \ref{lemma 2.11} (i) and Theorem \ref{thm 2.8}.
\begin{corollary}
Let $k$ be a positive integer. Any 2-factor of order $n$ with $n\geq 6k-4$ is $k$-placeable. Moreover, a cycle of order $n$ with $n\geq 2k+1$ is $k$-placeable for $k\geq 4$.
\end{corollary}

\section{Proof of Theorem \ref{theorem 2}} \label{section 4}

Let $G$ be an $(n,\,n-1)$-graph on  $n\geq 8$ vertices. If $G$ is connected, then Theorem \ref{lem 1111} implies that Theorem \ref{theorem 2} holds. Therefore, we may assume that $G$ is a disconnected $(n,\,n-1)$-graph, then $G$ has at least one cycle. Suppose that $g(G)\geq9$ and $\Delta(G)\leq n-4$. Clearly, $C_9$ is 4-placeable by Lemma \ref{lemma 2.11} (i). Thus we only need to prove sufficiency of Theorem \ref{theorem 2} with $n\geq 10$.

\begin{lem}\label{lemma 3.1} If \,$10\leq n\leq 13$, then $G$ is $4$-placeable.
\end{lem}

\begin{proof} Since $G$ is a disconnected $(n,\,n-1)$-graph with $g(G)\geq 9$ and $n\leq 13$, $G$ does not contain $D(l,s,t)$ as a subgraph. Moreover, $G$ has exactly two components $A$ and $B$, where $A$ contains a cycle $C_s$ ($s\geq 9$). If $A\not\cong C_s$, delete some leaves of $A$ to obtain $L(l,s)$ with $l$ maximum. Let $w\in V(L(l,s))$ with degree three. Moreover, we can get a lasso $L(l,s)$ from $A$ such that there exists $u\in N_{C_s}(w)$ with $d_G(u)=2$. Similarly, delete some leaves of $B$ to obtain a path $P_t$ with $t$ maximum.  Then add an edge between the leaf of $L(l,s)$ (or any vertex of $A$ if $A\cong C_s$) and a vertex of $P_t$ with degree at most one to get $L(l+t,s)$. Lemma \ref{lemma 2.11} (ii) implies that $L(l+t,s)$ has a 4-placement such that all vertices except  $u$ are 4-placed. Then $G$ has a 4-placement by Observation \ref{observation 1}.\qed\end{proof}

Suppose that $n\geq 14$. We prove Theorem \ref{theorem 2} by induction on $n$ and assume that Theorem \ref{theorem 2} holds for $(n^{\prime},n^{\prime}-1)$-graphs with $10\leq n^{\prime}<n$.  Now we consider the case  $v(G)=n$.

\begin{lem}\label{claim 3.111}
Let $A,B$ be two disjoint induced subgraphs of $G$ and $V(G)=V(A)\cup V(B)$. Suppose that $A$ consists of $x$ trees with $x\geq 3$ and it contains a vertex $u\in V(A)$ such that $E(u,B)=E(A,B)$. If $e(u,B)\geq x$, then $B$ has a $4$-placement in $K_{v(B)}$. \end{lem}

\begin{proof} Let $v(B)=l$ and $y=e(u,B)$. Recall that $G$ has a cycle and the girth of $G$ is at least 9, so $B$ is an $(l,\,l-1+x-y)$-graph with $g(B)\geq 9$. Also since $G$ contains a cycle, $A$ is a forest and $E(A,B)=E(u,B)$, we conclude that $B$ or $G[V(B)\cup\{u\}]$ contains a cycle with length at least 9. Thus there is an induced  path of order at least 8 in $B$. So $l\geq 8$. In fact, $l\geq 9$ by $y\geq x\geq 3$ and $g(G)\geq 9$. Suppose that $B\cong C^1\uplus C^2\uplus\cdots\uplus C^s$, where $C^i$ is a component of $B$.  Add edges $e_1, e_2, \ldots $ in turn between the components of $B$ until we obtain an $(l,\,l-1)$-graph $B^{\prime}$,  where $e_j=u_ju_{j+1}$ such that $u_j$, $u_{j+1}$ has minimum degree in $\uplus_{i=1}^j C^i$, $C^{j+1}$, respectively.

Obviously $g(B^{\prime})=g(B)\geq 9$ and $v(B^{\prime})=v(B)\geq 9$. In addition, $B^{\prime}$ also contains an induced  path with order at least 8. Thus $\Delta(B^{\prime})\leq l-4$.  Further, $B^{\prime}$ does not belong to $W$ (shown in Fig. \ref{fig1}) because each tree in $W$ does not contain a path of order at least 8. Then by the induction hypothesis, $B^{\prime}$, consequently $B$ has a $4$-placement in $K_{v(B)}$.\qed\end{proof}

\begin{lem}\label{lemma 3.2} If \,$G$ has four distinct leaves  such that they have distinct neighbors, then $G$ is $4$-placeable. Moreover, if $G$ contains four nodes, then $G$ is 4-placeable.
\end{lem}

\begin{proof} Let $U$ be a set of four leaves such that the vertices in $U$ has distinct neighbors and let $H=G-U$.  Clearly, $H$ is a disconnected $(n-4,\,n-5)$-graph with $g(H)\geq 9$, where $n-4\geq 10$. It follows by  $g(H)\geq 9$ that $\Delta(H)\leq v(H)-4$. Thus $H$ is 4-placeable  by the induction hypothesis. Moreover, Lemma \ref{lem 2} implies that $G$ is 4-placeable. Clearly, if $G$ has four nodes, then it has four distinct leaves  such that they have distinct neighbors. \qed\end{proof}

In the following,  we prove that if there are two components of $G$ each of which is a tree, then $G$ is $4$-placeable. First, we prove a useful claim.
\begin{clm}\label{clm 3.112} Suppose that $G$ is not 4-placeable. If there are three consecutive vertices on a path  with degree sequence  $S$ in $G$, then $G$ contains at most $x$ isolated vertices, where

$$x=\begin{cases}
1,& \text{if } S=(3,2,2),\\
2,& \text{if } S=(3,2,3), \text{ and }\\
3,& \text{if } S=(3,3,3).
\end{cases}$$\end{clm}

\begin{proof} Suppose the claim does not hold. Assume that $G$ contains $x+1$ isolated vertices $s_1,\ldots, s_{x+1}$, and let $u,v,w$ be three consecutive vertices of $G$  satisfying the specific degree sequence $S$. Suppose that $N_G(u)=\{u_1,u_2,v\}$,  $N_G(v)=\{v_1,u,w\}$ and $N_G(w)=\{w_1,w_2,v\}$, where $v_1,w_2$ may not exist.  Let $M=G[\{s_1,\ldots,s_{x+1},u,v,w\}]$ and $H=G-M$.  Divide $K_{n}$ into two disjoint subgraphs $K_{v(H)}$ and $K_{v(M)}$.

First we claim that $H$ has a 4-placement in $K_{v(H)}$.  Clearly, $H$ is a $(v(H),\,v(H)-1)$-graph with $g(H)\geq 9$. By the induction hypothesis, it suffices to show that $v(H)\geq 8$, $\Delta(H)\leq v(H)-4$ and $H\notin W$. It follows by $g(G)\geq 9$ that $H$ or $G[V(H)\cup \{u,v,w\}]$ contains an induced cycle with length at least 9. In the former case, since $g(G)\geq 9$, $H$ satisfies the three properties above clearly. Thus we may assume that the latter holds, that is, $H$ is a tree (recall that $H$ is a $(v(H),\,v(H)-1)$-graph) and $G[V(H)\cup \{u,v,w\}]$ contains an induced cycle with length at least 9. So there is an induced path in $H$ of order at least six and then $\Delta(H)\leq v(H)-4$ as $g(G)\geq 9$. Moreover, since $g(G)\geq 9$ and $|N_H(\{u,v,w\})|\geq 3$, we have that $v(H)\geq 8$ and  $H\notin W$ (shown in Fig. \ref{fig1}).

Thus, in order to obtain a 4-placement of $G$ (which contradicts the assumption that $G$ is not 4-placeable), it suffices to put four edge-disjoint copies of $M$ in $K_{v(M)}$ such that $\Phi(E(M,H))$ are edge disjoint. In fact, we only need to consider how to put 4 copies of $u,v$ and $w$ in $K_{v(M)}$ such that $\Phi(E(M,H))$ are edge disjoint: If $\phi_i(\{u,v,w\})$ is known in $K_{v(M)}$, then embed $s_1,\ldots, s_{x+1}$ arbitrarily in $K_{v(M)}-\{\phi_i(u),\phi_i(v),\phi_i(w)\}$. So, in the following, we omit $\phi_i(s_j)$ for $1\leq i\leq 4$, $1\leq j\leq x+1$. For convenience, let $V(K_{v(H)})=V(H)$ and $V(K_{v(M)})=V(M)$. In particular, we write $\phi_i(u,v,w)=(\phi_i(u),\phi_i(v),\phi_i(w))$ (an ordered 3-tuple).

\medskip
\noindent\textbf{Case 1. } $S=(3,2,2)$.
\medskip

Observe that the 4-placement of $w_1$ affects the 4-placement of $w$. In fact, if $w_1$ is 4-fixed (resp. 4-placed), then we can construct a 4-placement of $M$ such that $w$ is 4-placed (resp. 4-fixed). So we divided the proof into the following three cases. If $w_1$ is $4$-fixed, let $A=M$, $B=H$ and $U=\{w_1\}$. If $w_1$ is $4$-placed, let $A=G[V(M)-\{w\}]$, $U=\{w\}$ and $B=H$.  Lemmas \ref{lemma 2.7} (iii) and \ref{lem 4} imply that $G$ is $4$-placeable in both cases. So we may assume that $\phi_1(w_1)=p,\phi_2(w_1)=p,\phi_3(w_1)=q$ and $\phi_4(w_1)\in \{p,q,r\}$, where $p,q,r$ are three distinct vertices of $V(K_{v(H)})$. We define $\phi_i(u,v,w)$ as follows.

$$\left\{\begin{aligned}
  &(v,w,s_1),  \,\,\,i=1 \\
  &(s_1,v,u),  \,\,\, i=2 \\
  &(s_2,s_1,u),  \,\,\, i=3\\
  &(w,x,y), \mbox{ where }(x,y,\phi_4(w_1))\in\{(s_2,v,p),(u,s_2,q),(s_2,u,r)\}, i=4.
\end{aligned}
\right.$$

We can check that whatever $\Phi(\{u_1,u_2\})$ is, we can  get a $4$-placement of $G$ because $u$ is 4-placed and $\phi_i(N_H(u))\cap \phi_i(N_H(w))=\emptyset$ for each $1\leq i \leq 4$.

\medskip
\noindent\textbf{Case 2.} $S=(3,2,3)$ or  $S=(3,3,3)$.
\medskip

Recall that we already put four copies of $H$ in $K_{v(H)}$. For convenience, suppose $\phi_1(N_H(u))=\phi_1(\{u_1,u_2\})=\{u_1,u_2\}$ and $\phi_1(N_H(w))=\phi_1(\{w_1,w_2\})=\{w_1,w_2\}$ in $\Phi(H)$. If $v_1$ exists, suppose $\phi_1(N_H(v))=\phi_1(\{v_1\})=\{v_1\}$. We choose $l,t$ (and also $l^{\prime}, t^{\prime}$) to be some permutation of $u,w$, that is, $\{l,t\}=\{l^{\prime}, t^{\prime}\}=\{u,w\}$. Further, if $l=u$ (resp. $w$), we sometimes use $N_H(l)$ to denote $N_H(u)$ (resp. $N_H(w)$) and use $l_1,l_2$  to denote $u_1,u_2$ (resp. $w_1,w_2$).  One can define $N_H(t),N_H(l^{\prime}),N_H(t^{\prime}),t_i,l^{\prime}_i$ and $t^{\prime}_i$ for $i=1,2$ in this way. Observe that if $\phi_\alpha(N_H(l))\cap \phi_\beta(N_H(l^{\prime}))=\emptyset$ for $1\leq \alpha\neq \beta\leq 4$, then we can put $\phi_\alpha(l)$ and $\phi_\beta(l^{\prime})$ on a same vertex of $K_{v(M)}$.

In fact, such $\alpha,\beta,l,l^{\prime}$ exist. If not, for each $l\in\{u,w\}$ and each $2\leq  \alpha\leq 4$, we have $\phi_\alpha(N_H(l))\cap \phi_1(N_H(u))\neq \emptyset$ and $\phi_\alpha(N_H(l))\cap \phi_1(N_H(w))\neq \emptyset$ ($\beta=1$). Then $\{\phi_\alpha(N_{H}(u)),\phi_\alpha(N_H(w))\}$ is $\{\{u_1,w_1\},\{u_2,w_2\}\}$ or $\{\{u_1,w_2\},\{u_2,w_1\}\}$ for each $\alpha\in\{2,3,4\}$. So there exist $\alpha\neq \beta\in\{2,3,4\}$ such that $\{\phi_\alpha(N_{H}(u)),\phi_\alpha(N_H(w))\}$=$\{\phi_\beta(N_{H}(u)),\phi_\beta(N_H(w))\}$. Then one can choose $l,l^{\prime}\in\{u,w\}$ satisfying $\phi_\alpha(N_H(l))\cap \phi_\beta(N_H(l^{\prime}))=\emptyset$ easily.  Choose $\alpha,\beta,l,l^{\prime}$  such that $$\phi_\alpha(N_H(l))\cap\phi_\beta(N_H(l^{\prime}))=\emptyset \text{ and then, } |\phi_\alpha(N_H(t))\cap\phi_\beta(N_H(l^{\prime}))| \text{ is maximum.} \qquad(\ast)$$
Without loss of generality, assume that $\alpha=1$ and $\beta=2$.

\medskip
\noindent\textbf{Subcase 2.1. $S=(3,2,3)$.}
\medskip

Note that in this case, the vertex $v$ has no neighbor in $H$. First put $\phi_1(l)$ and $\phi_2(l^{\prime})$ on $u$  in $K_{v(M)}$ as $\phi_1(N_H(l))\cap\phi_2(N_H(l^{\prime}))=\emptyset$. If $\phi_1(N_H(t))\cap \phi_2(N_H(t^{\prime}))=\emptyset$, then put $\phi_1(t)$ and $\phi_2(t^{\prime})$ on $w$ (See Fig. \ref{fig10} (a)). Further, one may get a 4-placement of $G$ by choosing $$\phi_1(v)=v,\phi_2(v)=s_1, \phi_3(u,v,w)=(s_2,u,s_3),\text{ and } \phi_4(u,v,w)=(s_1,s_2,v).$$
 This is a contradiction to the assumption that $G$ is not 4-placeable. So in the following, we may assume that $t_1\in \phi_1(N_H(t))\cap \phi_2(N_H(t^{\prime}))$ (recall that $\phi_1(N_H(t))=\{t_1,t_2\}$). Further, let $\phi_1(l,v,t)=(u,v,w) \text{ and } \phi_2(l^{\prime},v,t^{\prime})=(u,s_1,v).$

If $\phi_{i}(N_H(y))\cap \phi_1(N_H(t))=\emptyset$ for some $y\in\{u,w\}$ and some $i\in\{3,4\}$, then let $\phi_{i}(y)=\phi_1(t)=w$, $\phi_i(v)=u$ and $\phi_i(\{u,w\}-\{y\})=s_2$. And, set $\phi_j(u,v,w)=(s_1,s_2,s_3)$, where $\{i,j\}=\{3,4\}$ (see Fig. \ref{fig10} (b)). Therefore, $\phi_i(N_H(y))\cap \{t_1,t_2\}\neq \emptyset$ for every $y\in\{u,w\}$ and every $i\in\{3,4\}$. Assume that $\{\phi_3(N_H(u)),\phi_3(N_H(w))\}=\{\{t_1,a\},\{t_2,b\}\}$ and $\{\phi_4(N_H(u)),\phi_4(N_H(w))\}=\{\{t_1,c\},\{t_2,d\}\}$, where $a,b,c,d\in V(H)-\{t_1,t_2\}$, $a\neq b$ and $c\neq d$.

\begin{figure}[H]
\centering    
    \includegraphics[scale=0.35]{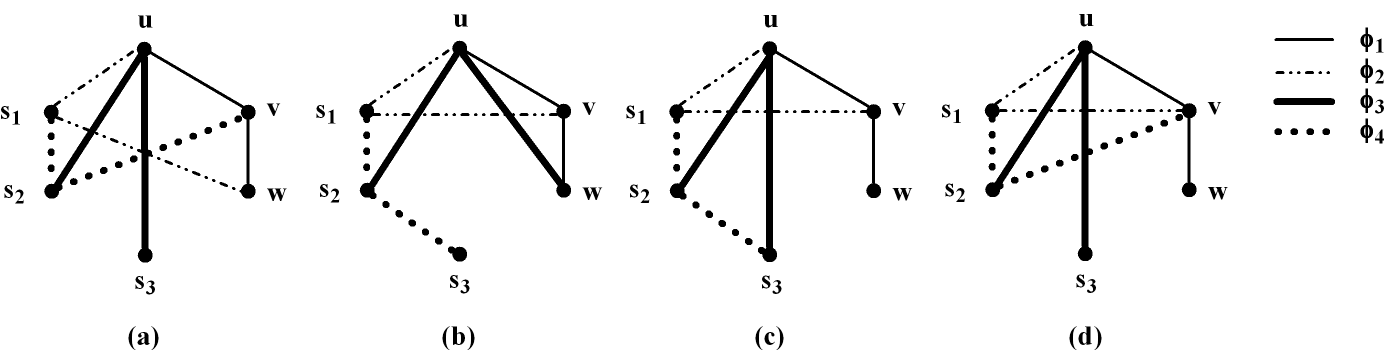}
    \caption{The subcases of the case $S=(3,2,3)$.}\label{fig10}
\end{figure}

If there exist $y,z\in\{u,w\}$ such that $\phi_3(N_{H}(y))\cap \phi_4(N_H(z))=\emptyset$, then let  $\phi_3(y,v,\{u,w\}-\{y\})=(s_3,u,s_2)$ and $\phi_4(z,v,\{u,w\}-\{z\})=(s_3,s_2,s_1)$ (see Fig. \ref{fig10} (c)). It is not difficult to check that a 4-placement of $G$ is obtained in this way, contradicting the assumption. Thus $\phi_3(N_{H}(y))\cap \phi_4(N_H(z))\neq \emptyset$ for each $y,z\in\{u,w\}$. Then, more precisely, we may assume that $$\{\phi_3(N_H(u)),\phi_3(N_H(w))\}=\{\{t_1,a\},\{t_2,b\}\} \text{ and } \{\phi_4(N_H(u)),\phi_4(N_H(w))\}=\{\{t_1,b\},\{t_2,a\}\}.$$ Recall that $t_1\in \phi_2(N_H(t^{\prime}))$ and $|N_H(t^{\prime})|=2$. Then $\{t_2,a\}$ or $\{t_2,b\}$ does not intersect with $\phi_2(N_H(t^{\prime}))$. (Observe that $t_2\notin \phi_2(N_H(t^{\prime}))$ as otherwise $\phi_2(N_H(t^{\prime}))=\{t_1,t_2\}$. In this case since $\phi_2(N_H(l^{\prime}))\cap\phi_2(N_H(t^{\prime}))=\emptyset$ and $\phi_2(N_H(t^{\prime}))=\phi_1(N_H(t))=\{t_1,t_2\}$, we obtain that $|\phi_2(N_H(l^{\prime}))\cap \phi_1(N_H(t))|=0$. Then we obtain a contradiction with the choice ($\ast$) as $\phi_2(N_H(t^{\prime}))\cap \phi_1(N_H(l))=\emptyset$ and  $|\phi_2(N_H(t^{\prime}))\cap \phi_1(N_H(t))|=2>|\phi_2(N_H(l^{\prime}))\cap \phi_1(N_H(t))|$). Suppose $\phi_i(N_H(y))=\{t_2,b\}$ with $\{t_2,b\}\cap \phi_2(N_H(t^{\prime}))=\emptyset$ for some $i\in\{3,4\}$ and some $y\in\{u,w\}$. Let  $\phi_i(y,v,\{u,w\}-\{y\})=(v,s_2,s_1)$ and $\phi_j(u,v,w)=(s_2,u,s_3)$, where $\{i,j\}=\{3,4\}$ (see Fig. \ref{fig10} (d)). We get a 4-placement of $G$, a contradiction again.

\medskip
\noindent\textbf{Subcase 2.2. $S=(3,3,3)$.}
\medskip

In this case, $|N_H(u)|=|N_H(w)|=2$ and $|N_H(v)|=1$. Roughly speaking, if we put $\phi_1(M), \phi_2(M)$ and $\phi_3(M)$ on $K_{v(M)}$ properly, then we may put $\phi_4(M)$ easily. More precisely,  we claim that after putting $\phi_1(M),\phi_2(M)$ and $\phi_3(M)$ on $K_{v(M)}$, if $a,b$ and $c$ are three independent vertices in $K_{v(M)}$ with $e(a,H)\leq 2$, $e(b,H)\leq 1$ and $e(c,H)=0$, then one can put $\phi_4(\{u,v,w\})$ easily on $\{a,b,c\}$. Set $\{r,y,z\}=\{u,v,w\}$, since $e(a,H)\leq 2$, $|N_H(u)|=|N_H(w)|=2$ and $|N_H(v)|=1$, there is a vertex in $\{r,y,z\}$, say $r$, such that $\phi_{4}(N_H(r))\cap N_H(a)=\emptyset$ and then let $\phi_4(r)=a$. Also, since $e(b,H)\leq 1$, there is a vertex in $\{y,z\}$, say $y$, such that  $\phi_4(N_H(y))$ does not intersect with $N_H(b)$. So we may let $\phi_4(y)=b$  and $\phi_4(z)=c$.

\begin{figure}[H]
\centering    
    \includegraphics[scale=0.35]{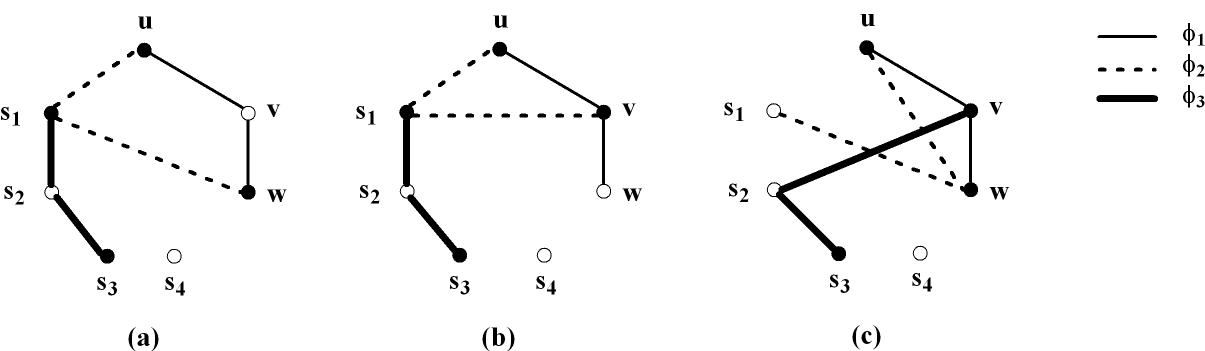}
    \caption{The case of $S=(3,3,3)$.}\label{fig11}
\end{figure}

First put $\phi_1(l)$ and $\phi_2(l^{\prime})$ on $u$ in $K_{v(M)}$ as $\phi_1(N_H(l))\cap\phi_2(N_H(l^{\prime}))=\emptyset$. If $\phi_1(N_H(t))\cap \phi_2(N_H(t^{\prime}))=\emptyset$, then let $\phi_1(l,v,t)=(u,v,w)$, $\phi_2(l^{\prime},v,t^{\prime})=(u,s_1,w)$ and $\phi_3(u,v,w)=(s_1,s_2,s_3)$. Now we claim that we may put $\phi_3(u)$ on $s_1$. Since $|N_H(v)|=1$, we get that $\phi_3(N_H(u))\cap \phi_2(N_H(v))=\emptyset$ or $\phi_3(N_H(w))\cap \phi_2(N_H(v))=\emptyset$. Here, we assume that the former holds. If the latter holds, then swap $\phi_3(u)$ and $\phi_3(w)$, that is, let $\phi_3(u,v,w)=(s_3,s_2,s_1)$. After putting $\phi_1(M)$, $\phi_2(M)$ and $\phi_3(M)$, we have  $e(v,H)=e(s_2,H)=1$ and $e(s_4,H)=0$, so we can put $\phi_4(\{u,v,w\})$ on $\{v,s_2,s_4\}$ easily (see Fig. \ref{fig11} (a)).

Thus we may assume  $t_1\in \phi_1(N_H(t))\cap \phi_2(N_H(t^{\prime}))$. If $v_1\notin\phi_2(N_H(t^{\prime}))$, then adjust $\phi_{2}(t^{\prime})$ to $\phi_1(v)$ (i.e. to $v$) on the basis of Fig. \ref{fig11} (a) and,  put $\phi_4(\{u,v,w\})$ on $\{s_2,s_4,w\}$ (see Fig. \ref{fig11} (b)). So it suffices to consider the case that $\phi_2(N_H(t^{\prime}))=\{t_1,v_1\}$. In this case, we have $\phi_2(N_H(t^{\prime}))\cap \phi_1(N_H(l))=\{t_1,v_1\}\cap \{l_1,l_2\}=\emptyset$ and $|\phi_2(N_H(t^{\prime}))\cap \phi_1(N_H(t))|=|\{t_1\}|=1$. It follows from the choice ($\ast$) that $|\phi_2(N_H(l^{\prime}))\cap \phi_1(N_H(t))|\geq 1$. Note that $\phi_2(N_H(l^{\prime})),\phi_2(N_H(v))$ and $\phi_2(N_H(t^{\prime}))$ are pairwise disjoint. So $t_2\in \phi_2(N_H(l^{\prime}))$ and  $\phi_2(N_H(v))\cap \phi_1(N_H(t))=\emptyset$. Let $\phi_2(l^{\prime},v,t^{\prime})=(u,w,s_1)$, $\phi_3(u,v,w)=(v,s_2,s_3)$ (here, we assume $\phi_3(N_H(u))\cap \phi_1(N_H(v))=\emptyset$, otherwise, similarly to the paragraph above, we swap $\phi_3(u)$ and $\phi_3(w)$) and put $\phi_4(\{u,v,w\})$ on $\{s_1,s_2,s_4\}$ (see Fig. \ref{fig11} (c)). A 4-placement of $G$ is obtained, contradicting the assumption.\qed\end{proof}

Let $T^1,\ldots, T^a$ denote the components of $G$ that are trees and let $C^{a+1},\ldots, C^b$ denote the components of $G$ that are not trees. Furthermore, say $v(T^1)\geq \cdots \geq v(T^a)$. Note that at least one component of $G$ must be a tree and recall that we may assume that $G$ has at least one cycle, so indeed $b > a \geq 1$.

\begin{lem}\label{lemma 3.3} If \,$G$ has at least two components each of which is a tree, then $G$ is $4$-placeable.
\end{lem}
\begin{proof} Since $G$ is an $(n,\,n-1)$-graph,  we have $\Delta(G-(T^1\uplus T^2))\geq 3$, say $u\in V(G-(T^1\uplus T^2))$ is a vertex with maximum degree. If $v(T^1\uplus T^2)\geq 3$, let $A=G[V(T^1\uplus T^2)\cup \{u\}]$, $B=G-A$. By Lemma \ref{lemma 3.2}, $T^1\uplus T^2$ contains at most three leaves with distinct neighbors, otherwise we are done. Then deleting some leaves of $T^1\uplus T^2$ if necessary, we obtain $K_1\uplus P_{l_1}\uplus P_{l_2}$ with $l_1+l_2\geq 3$ or $2K_1\uplus Q(n_1,n_2,n_3)$ with $n_1\geq 2$ from $A$. Lemma \ref{lemma 2.7} (iii), (iv) and Observation \ref{observation 1} imply that $A$ has a 4-placement such that $u$ is 4-placed. Moreover, $B$ has a 4-placement by Lemma \ref{claim 3.111}. Thus $G$ is 4-placeable by Lemma \ref{lem 4} ($U=\emptyset$) and the lemma holds. Thus $T^1\cong\cdots\cong T^a\cong K_1$. That is, each tree $T^i$ in $G$ is in fact an isolated vertex.

First we consider the case that the number of isolated vertices in $G$ is at least three, i.e., $a\geq 3$.  In this case, we claim that $\Delta(G)=3$. If not, assume $d_G(u)\geq 4$ and let $A=G[V(T^1\uplus T^2\uplus T^3)\cup \{u\}]$ and  $B=G-A$. Clearly, $A$ has a 4-placement such that $u$ is 4-placed. Moreover, Lemma \ref{claim 3.111} implies that $B$ has a 4-placement, then $G$ is 4-placeable by Lemma \ref{lem 4} ($U=\emptyset$). We are done.

Further, since $G$ is an $(n,\,n-1)$-graph, there is a component $C^i$ of $G$ with at least $v(C^i)+1$ edges. Then $C^i$ contains a double lasso $L$ as a subgraph.  It follows by $\Delta(G)=3$ that $\Delta(C^i)=3$. Suppose $v\in V(L)$ with $d_G(v)=3$. Since each vertex on the double lasso $L$ has degree at least 2 in $G$, one may find three consecutive vertices on $L$ with degree sequence (in $G$) either $(3,2,2)$, or $(3,2,3)$, or  $(3,3,3)$. By Claim \ref{clm 3.112}, we have that $a=3$ and every three consecutive vertices of $L$ have degree sequence  ($3,3,3$), otherwise we are done. In other words, each vertex on $L$ has degree three in $G$.

Let $x_i$ be the number of vertices in $G$ with degree $i$ for $1\leq i\leq 3$. Since $G$ is an $(n,\,n-1)$-graph and $\Delta(G)=3$ and $a=3$, we obtain that $3+x_1+x_2+x_3=n$ and $x_1+2x_2+3x_3=2n-2$. Thus $x_3=4+x_1$. Lemma \ref{lemma 3.2} implies that the number of nodes of $G$ is at most three, then $x_1\leq 6$ because $\Delta(G)=3$ and $G$ does not  contain non-trivial tree. That is, $x_3\leq 10$. However, it is easy to check that $v(L)>10$ as $g(G)\geq 9$, a contradiction.

Now we consider the case $a=2$, that is, $G$ contains exactly two trees (isolated vertices) as components. In this case, the structure of $G$ can be easily described: $G\cong 2K_1\uplus C^3\uplus\cdots\uplus C^b$, where $C^3$ is a $(v(C^3),v(C^3)+1)$-graph and $C^i$ is a $(v(C^i),v(C^i))$-graph for each $4\leq i\leq b$.  Clearly, $C^3$ contains a double lasso $L$ as a subgraph. Moreover, if $C^3\cong L$, then by Lemma \ref{lemma 2.11} (iii), $C^3$ is 4-placeable. By the induction hypothesis, $G-C^3$ is also 4-placeable (adding an edge between two isolated vertices one may get an $(l,l-1)$-graph for some $l$). Then $G$ is 4-placeable, the lemma holds.

\begin{figure}[H]
\centering    
    \includegraphics[scale=0.3]{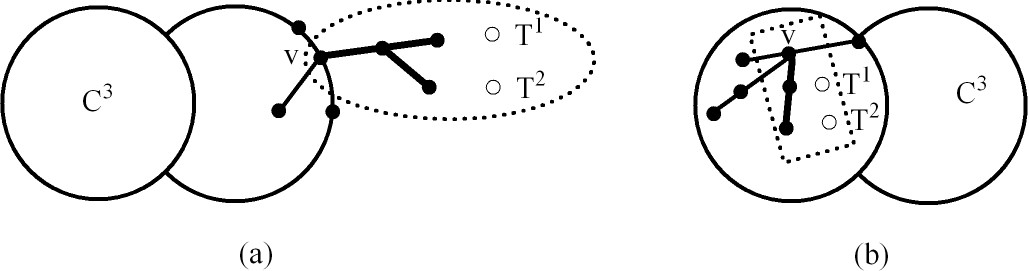}
    \caption{The case of $a=2$ in Lemma \ref{lemma 3.3}, where the white vertices are $T^1$ and $T^2$. The bold lines indicate the subgraph $T$.}\label{fig12}
\end{figure}

Therefore, in $C^3$, there are some trees each of which intersects with the double lasso $L$ at a leaf. Moreover, any two trees can only intersect at such leaves, and some of these trees may intersect $L$  at a same vertex, see Fig. \ref{fig12} (a).  Now we claim that $\Delta(C^3)=3$. Otherwise, choose $v\in V(C^3-L)$ with $d_{G}(v)\geq 4$ unless all such vertices are in $L$, e.g. see Fig. \ref{fig12} (b). Let $A=T^1\uplus T^2\uplus T$ and $B=G-A$, where $T$ is a subgraph of $C^3$ containing $v$ and  $d_{T}(v)=1$. It follows by Lemma \ref{lemma 3.2} and the choice of $v$ that $T$ can be obtained by adding some leaves to a path or $Q(n_1,n_2,n_3)$. By Lemma \ref{lemma 2.7} (iii)-(iv), Observation \ref{observation 1} and Claim \ref{claim 3.111}, we see that $A$ and $B$  has a 4-placement, respectively.  Then $G$ is 4-placeable by Lemma \ref{lem 4}($U=\emptyset$), we are done. Thus $2\leq d_G(w)\leq 3$ for each vertex $w\in V(L)$. Observe that there are at most three disjoint trees each of which intersects with $L$ at a leaf because each such tree contributes a node to $G$. Since $g(G)\geq 9$, it is not difficult to check that there are three consecutive vertices on $L$  with degree sequence $(3,2,2)$, which contradicts Claim \ref{clm 3.112}.\qed\end{proof}

We are in the position to prove  Theorem \ref{theorem 2}. First, we describe the structure of $G$. Recall that $G=T^1\uplus \cdots \uplus T^a\uplus C^{a+1}\uplus \cdots \uplus C^{b}$. Lemma \ref{lemma 3.3} implies that  $a=1$. Moreover, since $G$ is a disconnected $(n,n-1)$-graph, each component $C^i$ ($2\leq i\leq b$) is a cycle or a $(v(C^i),v(C^i))$-graph containing one cycle and some trees such that each tree intersects with the cycle at a leaf. Moreover, any two trees can only intersect at such leaves, and some of these trees may intersect the cycle at a same vertex (for convenience, when we say that some trees intersect with a cycle in the following, we assume that these trees satisfy this requirement).

Let $G^{\prime}$ be a subgraph obtained from $G$ by deleting some leaves. Inspired by Observation \ref{observation 1}, in order to get a $4$-placement of $G$, it suffices to construct a 4-placement of $G^{\prime}$ such that all neighbors of leaves of $G$ are 4-placed. We call such 4-placement of $G^{\prime}$ \emph{good}. More precisely, we delete some leaves from component $C^i$ (or $T^i$) of $G$ with the following priority.
\medskip

(i) If deleting some leaves from $C^i$ of $G$ we get a lasso $L(l,s)$ with $l>s$, then delete such leaves from $C^i$ so that $l$ is maximum;  If $T^i\cong K_2$, delete one leaf of $T^i$ to get $K_1$;
\medskip

(ii)  Let $U$ be the set of leaves of $C^i$ (or $T^i$).  For each $u\in V-U$, let $b_u=|\{v\in N_G(u): d_G(v)\geq 2\}|$.  That is, $b_u=|N_G(u)-U|$. If $b_u=0$  or $b_u\geq 2$, then we delete all leaves adjacent with $u$. Otherwise, we delete all leaves adjacent with $u$ except one. See Fig. \ref{fig13} for an example illustrating the deletion.

\begin{figure}[H]
\centering    
    \includegraphics[scale=0.35]{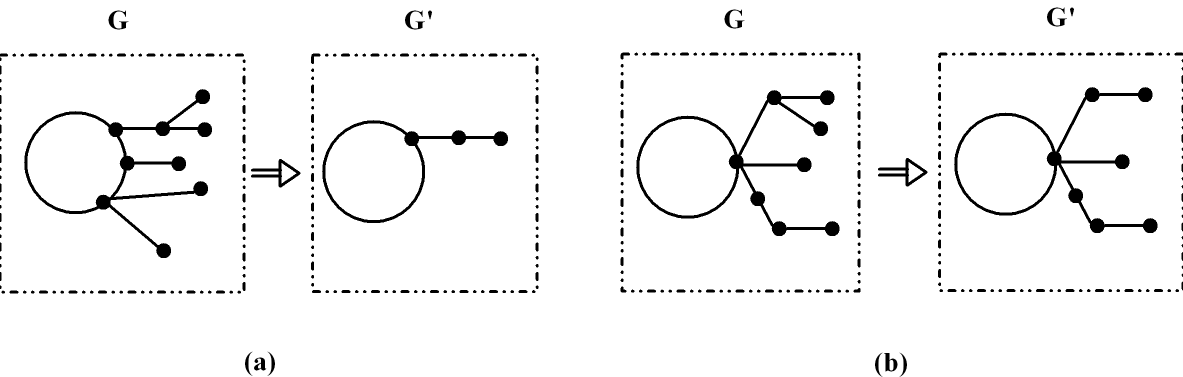}
    \caption{Examples of how $G^{\prime}$ can be obtained from $G$.}\label{fig13}
\end{figure}

It is not difficult to check that if a vertex is a node of $G^{\prime}$, then it is also a node of $G$ (this is the reason why we don't delete all the leaves). Clearly,  $G^{\prime}$ is a $(v(G^{\prime}),\,v(G^{\prime})-1)$-graph. In addition,  it has at most 3 nodes as otherwise $G$ has four nodes and then $G$ is 4-placeable by Lemma \ref{lemma 3.2}, we are done.  For convenience, we write $G^{\prime}=T^{\prime1}\uplus C^{\prime2}\uplus \cdots \uplus C^{\prime b}$. In fact, $T^{\prime1}$ and $C^{\prime i}$ are known as $T^{\prime1}\cong P_t$ ($t=1$ or $t\geq 4$) or $T^{\prime1}\cong Q(n_1,n_2,n_3)$ ($n_3\geq n_2\geq n_1\geq 2$), where $t=1$ if the component $T^{\prime1}$ of $G$ is a star (i.e. $b_u=0$, where $u$ is the center of the star) and $t\geq 4$ if it is a non-star;  Each $C^{\prime i}$ ($(v(C^{\prime i}),\,v(C^{\prime i}))$-graph) has one cycle and some trees such that each tree intersects with the cycle at a leaf.

Observe that if the component $C^{\prime i}$ of $G^{\prime}$ is a cycle, then $C^i$ must also be a cycle in $G$. It is worthwhile mentioning that if all nodes of $C^i$ are on the unique cycle of $C^i$, we will obtain a lasso $L(s+1,s)$ in $G^{\prime}$ rather  than a cycle $C_s$.

\medskip
\noindent\textbf{Case 1. } $T^{\prime1}\not \cong K_1$.
\medskip

First we consider the case that $T^{\prime1}\cong P_t$ with $t\geq 4$.  Note that $P_t$ ($t\geq 4$) has two nodes, then $C^{\prime2}\uplus \cdots\uplus C^{\prime b}$ has at most one node. That is, there is at most one lasso $C^{\prime2}$  and other components are all  cycles. Add an edge between $P_t$ and $C^{\prime2}$ to get a lasso $L$. Lemma \ref{lemma 2.11} (i) and (ii) imply that $G^{\prime}$ has a 4-placement such that all vertices are 4-placed except the vertex $v_1$ of the lasso $L$. Since $G$ has at most three nodes and two nodes of $P_t$ are also the nodes of $G$, we have that $v_1$ or $v_{s-1}$ of the lasso is not a node of $G$. By symmetry, we may assume that $v_1$ is not a node. Therefore, such a $4$-placement of $G^{\prime}$ is good. We are done.

Thus $T^{\prime1}\cong Q(n_1,n_2,n_2)$ with $n_3\geq n_2\geq n_1\geq 2$. Note $Q(n_1,n_2,n_2)$ has three nodes, so $C^{\prime2},\ldots,C^{\prime b}$ are all cycles. Suppose that $C^{\prime2}\cong C_s$ with $s\geq 9$. If $n_3=2$, then $G^{\prime}$ has a good 4-placement by Lemmas \ref{lemma 2.7} (ii) and \ref{lemma 2.11} (i). Thus $n_3\geq 3$. Label $Q(n_1,n_2,n_3)$ as defined and label $C^2$ clockwise with $u_1,u_2,\ldots, u_{s}$. Deleting $u_2$ from $C^{\prime2}$ and adding edges $v_{n_2}^2u_3$, $v_{n_3}^3u_1$, we obtain the lasso $L(n_1+n_2+n_3+s,n_2+n_3+s)$.

Lemma \ref{lemma 2.11} (ii) implies that the lasso has a 4-placement such that all vertices except for $v_1^3$ are 4-placed. Furthermore, for $1\leq i,j\leq 4$, $\phi_i(u_1)$, $\phi_j(u_3)$ are pairwise distinct by the construction of the 4-placement of a lasso (see the proof of Lemma \ref{lemma 2.11} (ii)) and (\ref{eq 2}). Thus deleting  edges $\phi_i(v_{n_2}^2u_3)$, $\phi_i(v_{n_3}^3u_1)$ and adding a vertex $u_2$, edges $u_2\phi_i(u_1)$, $u_2\phi_i(u_3)$ for each $1\leq i\leq 4$, we obtain a 4-placement of $C_s\uplus Q(n_1,n_2,n_3)$ such that all nodes of $Q(n_1,n_2,n_3)$ (such nodes of $G^{\prime}$ are also the nodes of $G$) are 4-placed. Then  $G^{\prime}$ has a good 4-placement by Lemma \ref{lemma 2.11} (i).

\medskip
\noindent\textbf{Case 2. } $T^{\prime1}\cong K_1$.
\medskip

Recall that each $C^{i}$ of $G$ has one cycle and some trees such that each tree intersects with the cycle at a leaf. Note that if every component $C^i$ of $G$ has at most one tree intersecting with the unique cycle of $C_i$, and each of them in $G^{\prime}$ is a lasso or a cycle, then Lemma \ref{lemma 2.11} (i) and (ii) imply that each of $H$ and $K_1\uplus H$ (is a subgraph of a lasso) has a  4-placement, where $H$ is a lasso or cycle. Here since all vertices of the path of a lasso are 4-placed, such a  4-placement of $G^{\prime}$ is good, we are done.

Thus if $G$ has three components $C^2, C^3$ and $C^4$ such that each $C^i$ ($2\leq i\leq 4$) has a tree intersecting with the unique cycle of $C_i$ (in this case, each  $C^{\prime i}$ with $2\leq i\leq 4$ in $G^{\prime}$ is a lasso because $G$ has at most three nodes), we are done.  So we only need to consider the following two cases:
\medskip

(a) $C^2$ of $G$ has $y$ ($1\leq y\leq 3$) trees intersecting with the unique cycle of it;
\medskip

(b) $C^2$ and $C^3$ of $G$ has $x$ and $y$ trees ($y\geq x$) intersecting with the unique cycle of them, respectively.
\medskip

In particular, $C^{\prime2}$ (resp. each of $C^{\prime2}$ and $C^{\prime3}$) of $G^{\prime}$ is not a lasso when $y=1$ (resp. $x=y=1$) by the argument above. Notice that in both cases if $G$ contains a cycle as a component, then $G$ is 4-placeable by Lemma \ref{lemma 2.11} (i) and the induction hypothesis. Thus $G$ (consequently, $G^{\prime}$) contains no cycle as a component. In the following, we give an ($A,U,B$)-structure of $G^{\prime}$ such that $G^{\prime}$ has a good 4-placement or an ($A,U,B$)-structure of $G$ directly, contradicting the assumption that $G$ is not 4-placeable.

In the case (b), we give an ($A,U,B$)-structure of $G^{\prime}$ as exhibited in Fig. \ref{fig3} (a), where each $A$ and $B$ consists of a lasso and a path of order at least one and $U=\emptyset$. Lemmas \ref{lemma 2.11} (ii) and \ref{lem 4} imply that $G^{\prime}$ has a good 4-placement as all vertices on the path of a lasso are 4-placed.  (Here, the components $C^{\prime2}, C^{\prime3}$ and the graph shown in Fig. \ref{fig3} (a) have roughly the `same' structure. More precisely, two trees may intersect with $C^{\prime3}$ at a same vertex, or there is only one tree ($Q(n_1,n_2,n_3)$)  intersecting with the unique cycle of $C^{\prime3}$.) Thus we may assume the case (a) holds.  Notice that $C^{\prime2}$ has a cycle, say $C_s$, and deleting $C_s$ from the $C^{\prime2}$, we get a forest $F$.

Let $M=\{u_1,u_2,\ldots,u_{v(M)}\}$ be a vertex set with $u_i\in V(C_s)$ such that $N_F(u_i)\neq \emptyset$. Clearly, $1\leq v(M)\leq 3$ as $G^{\prime}$ has at most 3 nodes. Thus we only need to prove the following two subcases.

\medskip
\noindent\textbf{Subcase 1. } $1\leq|M|\leq 2$.
\medskip

First, we consider $|M|=1$. In graph $G$, if $v(F)\geq 6$, then by the induction hypothesis, $G[V(F)\cup \{u_1\}\cup V(T^1)]$ (recall that $T^1\cong K_1$) has a 4-placement. Moreover, $P_{s-1}$ ($C_s-u_1$) has a 4-placement such that $\phi_i(p)$, $\phi_j(q)$ $(1\leq i\neq j\leq 4)$ are pairwise distinct by Lemma \ref{lemma 2.1} and (\ref{eq 2}), where $p,q$ are end-vertices of $P_{s-1}$. Adding edges $\phi_i(p)\phi_i(u_1)$ and $\phi_i(q)\phi_i(u_1)$ for each $1\leq i\leq 4$, we obtain a 4-placement of $G$, we are done. So assume $v(F)\leq 5$. Moreover, if $C^{\prime2}$ in $G^{\prime}$ is a lasso, then by the argument in the paragraph below Case 2, we are done.  So by the way of deleting leaves we may assume that $G^{\prime}$ is isomorphic to the graph in Fig. \ref{fig3} (b) or the structure of $G^{\prime}$ is like  the graph in Fig. \ref{fig3} (c) (i.e. $B$ may be $K_1\uplus P_2\uplus P_3$ or $A$ may be a cycle rather than a lasso).

We construct an $(A,U,B)$-structure of $G^{\prime}$ as follows in these two cases: let $A=C_s$ (or $L(s+1,s)$), $U=\emptyset$ and $B$ consists of at least two paths and $K_1$ (see Fig. \ref{fig3} (b)-(c)). Lemmas \ref{lemma 2.11} (i), (ii) and \ref{lemma 2.7} (iii) imply that each of $A$ and $B$ has a 4-placement such that $u_1$ in Fig. \ref{fig3} (b) or the vertices on the path of the lasso ($A$) in Fig. \ref{fig3} (c) are 4-placed. Then $G^{\prime}$ has a good 4-placement by Lemma \ref{lem 4}. Here, since $G$ has at most three nodes, $v_1$ or $v_{s-1}$ of the lasso is not a node of $G$. By symmetry, we may assume that $v_1$ is not a node.
\begin{figure}[H]
\centering    
\subfigure[] 
{
    \begin{minipage}{2.8cm}
    \centering        
    \includegraphics[scale=0.3]{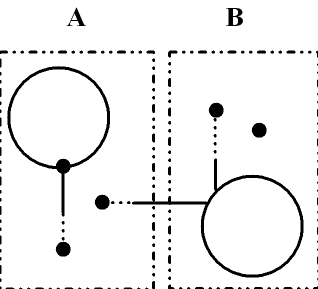}   
    \end{minipage}
}
\subfigure[]
{
    \begin{minipage}{2.8cm}
    \centering      
    \includegraphics[scale=0.3]{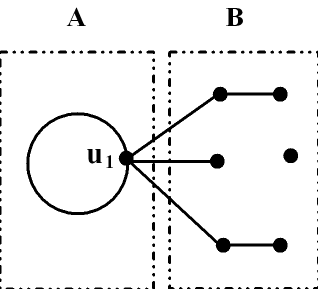}   
    \end{minipage}
}
\subfigure[]
{
    \begin{minipage}{2.8cm}
    \centering      
    \includegraphics[scale=0.3]{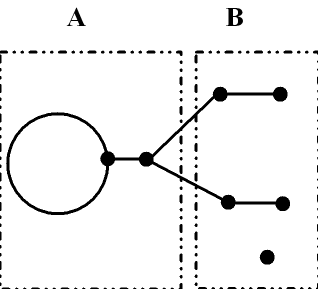}   
    \end{minipage}
}
\subfigure[]
{
    \begin{minipage}{2.8cm}
    \centering      
    \includegraphics[scale=0.3]{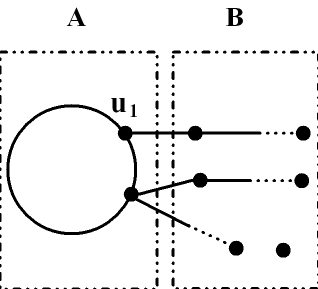}   
    \end{minipage}
}
\subfigure[]
{
    \begin{minipage}{2.8cm}
    \centering          
    \includegraphics[scale=0.3]{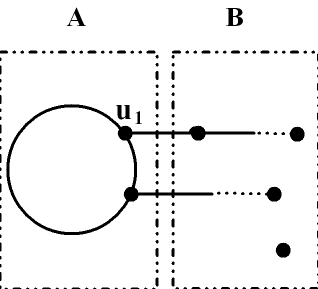}   
    \end{minipage}
}
\caption{($A,U,B$)-structures of $G^{\prime}$}
\label{fig3}  
\end{figure}
If $|M|=2$, then the structure of $G^{\prime}$ is like  one of the graphs in Fig. \ref{fig3} (d)-(e). Let $A$ be the cycle $C_s$, $U=\{u_1\}$ and $B=G-A$. Lemmas \ref{lemma 2.11} (i) ($u_1$ is 4-fixed by the construction of the 4-placement of a cycle), \ref{lemma 2.7} (iii) and  \ref{lem 4} imply that $G^{\prime}$ has a good 4-placement. Notice that if $u_1$ in Fig. \ref{fig3} (d) is a node, then we delete the leaves of $u_1$ and by the similar discussion of Fig. \ref{fig3} (c), $G$ has a 4-placement, since all vertices  of $\Phi(C_s)$ except one are 4-placed.

\medskip
\noindent\textbf{Subcase 2. } $|M|=3$, i.e. $M=\{u_1,u_2,u_3\}$.
\medskip

In this case, $F$ consists of three vertex disjoint paths, say $P^i=u_1^iu_2^i\cdots u_{n_i}^i$ $(1\leq i\leq 3)$ with $n_3\geq n_2\geq n_1$, where $u_iu_1^i\in E(G^{\prime})$. By the way of deleting leaves and the fact that $C^{\prime2}$ in $G^{\prime}$ is not a lasso, there are at least two nodes not on the cycle $C_s\in C^{\prime2}$. That is, $n_3\geq n_2\geq 2$. Let $A=L(s+n_1,s)$, $U=\{u_1^2\}$ and $B=K_1\uplus G^{\prime}[V(P^3\uplus P^2)-\{u_1^2\}]$. By Lemma \ref{lemma 2.11} (ii), we may get a 4-placement of $A$ such that $u_2,u_3$ are 4-placed.  Further by Lemmas \ref{lemma 2.7} (iii) and \ref{lem 4}, $G^{\prime}$ has a good 4-placement, we are done. \qed

\section*{Acknowledgments} \label{section 5}

We are very grateful to the referees for their many valuable suggestions and comments, which make the proof much simpler and clearer.

\end{document}